\documentclass{amsart}

\usepackage{amsfonts,amssymb,amscd,amsmath,enumerate,verbatim,calc,txfonts }
\usepackage{enumitem}
\usepackage{graphics,graphicx}
\usepackage[small,bf]{caption}
\usepackage[initials]{amsrefs}

\newcommand{\atoms}[1]{\mathcal{A}_{#1}}

\newcommand{\nonz}[1]{#1^*}
\newcommand{\units}[1]{U(#1)}
\newcommand{\nunits}[1]{#1^\#}

\newcommand{\MCD}{\mathrm{MCD}}
\newcommand{\GCD}{\mathrm{GCD}}

\newcommand{\assoc}[1]{\sim_{#1}}
\let\oldfrac\frac
\renewcommand{\frac}[2]{%
  \mathchoice
    {\oldfrac{#1}{#2}}
    {#1/#2}
    {#1/#2}
    {#1/#2}
}

\newcommand{\aadd}[1]{\mid_{#1}}

\newcommand{\lc}[1]{\mathrm{lc}(#1)}

\newcommand{\mybar}[1]{\overline{#1}}

\newtheorem{theorem}{Theorem}[section]
\newtheorem{lemma}[theorem]{Lemma}

\newtheorem{remark}[theorem]{Remark}

\theoremstyle{definition}

\theoremstyle{example}
\newtheorem{example}[theorem]{Example}

\begin{document}
\author[S. Eftekhari and M.R. Khorsandi]
{Sina ~Eftekhari  and  Mahdi Reza ~Khorsandi$^*$}

\title[MCD-finite Domains]
{MCD-finite Domains and Ascent of IDF Property in Polynomial Extensions }
\subjclass[2010]{13B25, 13F15, 13F20, 13G05} \keywords{Factorization, IDF, Irreducible, MCD, Polynomial Extensions}
\thanks{$^*$Corresponding author}
\thanks{E-mail addresses: khorsandi@shahroodut.ac.ir and sina.eftexari@gmail.com}
\maketitle

\begin{center}
{\it Faculty of Mathematical Sciences, Shahrood University of Technology,\\
P.O. Box 36199-95161, Shahrood, Iran.}
\end{center}
\vspace{0.4cm}
\begin{abstract}
An integral domain is said to have the IDF property when every non-zero element of it has only a finite number of non-associate irreducible divisors. A counterexample has already been found showing that IDF property does not necessarily ascend in polynomial extensions. 

In this paper, we introduce a new class of integral domains, called MCD-finite domains, and show that for any domain $D$,  $D[X]$ is an IDF domain if and only if $D$ is both IDF and MCD-finite. This result entails all the previously known sufficient conditions for the ascent of the IDF property. 

Our new characterization of polynomial domains with the IDF property enables us to use a different construction and build another counterexample which strengthen the previously known result on this matter.
\end{abstract}
\vspace{0.5cm}

\section{Introduction}
An integral domain is called IDF if every non-zero element of it has only a finite number of non-associate irreducible divisors. These domains were introduced by Grams and Warner in \cite{Grams1975}. They are also one of the generalizations of UFD's that were studied in the seminal paper \cite{anderson1990factorization}. Another important subclass of IDF domains are domains that contain no atoms at all; these domains were named antimatter and studied in \cite{Coykendall1999antimatter}. The class of IDF domains also includes FFD's. An integral domain is an FFD if every non-zero element of it has only a finite number of non-associate divisors. If $D$ is an FFD, then $D[X]$ is also an FFD (see \cite[Proposition~5.3]{anderson1990factorization}). A question first posed in \cite{anderson1990factorization} is whether the IDF property ascends in polynomial extensions. Malcolmson and Okoh in \cite{malcolmson2009polynomial}, answered this question in the negative. They actually proved that:

\begin{theorem}[{{\cite[Theorem~2.5]{malcolmson2009polynomial}}}]
\label{MOTheorem}
Every countable domain can be embedded in a countable antimatter domain $R_{\infty}$ such that $R_{\infty}[X]$ is not an IDF domain.
\end{theorem}

A natural question that follows is that under which additional conditions the IDF property does ascend in polynomial extensions. One such condition is when the domain, in addition to being IDF, is a valuation domain or more generally a GCD domain (see \cite[p.~14]{anderson1990factorization} and \cite[Theorem~1.9]{malcolmson2009polynomial}).
 Another case is when the domain is atomic. In fact, atomic IDF domains are exactly FFD's (see \cite[Theorem~5.1]{anderson1990factorization}). In Theorem \ref{mainTheorem}, we see that the essential property of these domains sufficient for the IDF property to ascend is that any finite set of non-zero elements of such domains has only a finite (possibly zero) number of non-associate maximal common divisors (MCD for short); in this paper, we call any such domain MCD-finite. Actually, Theorem \ref{mainTheorem} shows that being MCD-finite is also a necessary condition for the ascent of IDF property. In Theorem \ref{MOstrong}, we  use a modified version of a technique originally introduced by Roitman in \cite{roitman1993polynomial} to get a stronger version of Theorem \ref{MOTheorem}.  In Section 3, we provide some more results and examples regarding MCD-finite domains. 
 
 In the remainder of this section, we state some terminology that are needed for the rest of the paper.

For a domain $D$, the set of its non-zero elements, units and  non-zero non-units  are denoted by $\nonz{D}$, $\units{D}$ and $\nunits{D}$, respectively. We call two elements $x, y \in D$ associates, and write $x \sim y$, if there exists $u \in \units{D}$ such that $a = ub$. For $x,y \in D$, we say that $x$ divides $y$, and write $x \mid y$, if there exists $z \in D$ such that $y = xz$. Also, we use the notations $a \assoc{D} b$ and $a \aadd{D} b$ to emphasize the underlying domain. Two elements $a, b \in \nonz{D}$ are called incomparable if $a \nmid b$ and $b \nmid a$ (i.e., the principal ideals $\langle a \rangle$ and $\langle b \rangle$ are incomparable). Also, following the notation and terminology of \cite{rand2015multiplicative}, the set of atoms of $D$ is denoted by $\atoms{D}$ and we call a subset $A$ of $D$ unit-closed if for every $u \in \units{D}$ and $a \in A$, we have $ua \in A$. 

A domain $D$ is called atomic if every $x \in \nunits{D}$ can be written as a product of irreducible elements (atoms). If the set of principal ideals of $D$ satisfies the ascending chain condition, then $D$ is called an ACCP domain. We refer the reader to \cite{anderson1990factorization} for more on these and other domains with factorization properties.

Let $D$ be a domain and $S \subseteq \nonz{D}$. The set of all the common divisors of $S$ is denoted by
$\mathrm{CD}_D(S)$ (i.e., $\mathrm{CD}_D(S) = \{x \in \nonz{D} \colon\; x \mid s \; \text{for every} \; s \in S \}$). An element $c$ of $\mathrm{CD}_D(S)$ is called a maximal common divisor of $S$ whenever for every $d \in \mathrm{CD}_D(S)$, if $c \mid d$, then $c \sim d$. The set of all the maximal common divisors
of $S$ is denoted by $\MCD_D(S)$. As in \cite{roitman1993polynomial}, if for every finite subset $T$ of $\nonz{D}$, we have $\MCD_D(T) \neq \emptyset$, then $D$ is called an MCD domain. We say $D$ is MCD-finite if for every finite subset $T$ of $\nonz{D}$, we have $|\MCD_D(T)| < \infty$ (with the possibility $|\MCD_D(T)|=0$).

A commutative (additive) monoid $T$ is called cancellative if for all $a, b, c \in T$, if $a + b = a + c$, then $b = c$. Also, $T$ is  torsion-free if for all $a, b \in T$ and $n \in \mathbb{N}$, if $na = nb$, then $a = b$. Finally, $T$ is reduced if it has no non-trivial units; i.e., if $a + b = 0$ for some $a, b \in S$, then $a = b = 0$. All the concepts defined in the two previous paragraphs can be extended to cancellative monoids in an obvious way. We also recall that the monoid ring $R[X;T]$ is a domain if and only if $T$ is cancellative and torsion-free and $D$ is a domain (see \cite[Theorem~8.1]{gilmer1984commutative}).

We say that an extension of domains $A \subseteq B$ is division-preserving if for all $x, y \in \nonz{A}$,
if $x \aadd{B} y$, then $x \aadd{A} y$. This can also be stated concisely by saying $B \cap K = A$ where $K$ is the field of fractions of $A$ (For some properties of these extensions that are relevant to this paper, see \cite[Remark~2.2]{roitman1993polynomial}).

Finally, we recall some definitions from \cite{herzog2011monomial}. Let $D$ be a domain, $R = D[X_1, \dotsc , X_n]$ and $f \in R$. Any product $X_1^{k_1} \dotsm X_n^{k_n}$ with $k_i \in \mathbb{N} \cup \{0\}$ is called a monomial. If we denote the set of all the monomials of $R$ by $\mathrm{Mon}(R)$, then $f = \sum_{u \in \mathrm{Mon}(R)}a_u u$ for some elements $a_u \in D$. The support of $f$ is defined as the set $\{u \in \mathrm{Mon}(R)\colon\; a_u \neq 0\}$. By a monomial order on $R$, we mean a total order $<$ on the set $\mathrm{Mon}(R)$ such that (1) for every $u \in \mathrm{Mon}(R)$, if $u \neq 1$, then $1 < u$, and (2) if $u, v \in \mathrm{Mon}(R)$ and $u < v$, then $uw < vw$ for all $w \in \mathrm{Mon}(R)$. For example, we can consider the pure lexicographic order (induced by the ordering $X_1 > \dotsb > X_n$), where $X_1^{a_1} \dotsm X_n^{a_n} < X_1^{b_1} \dotsm X_n^{b_n}$ if the leftmost nonzero component of $(a_1-b_1,\dotsc, a_n-b_n)$ is negative (see \cite[Example~2.1.2(c)]{herzog2011monomial}). For a fixed monomial order, the leading coefficient of $f$, denoted by $\lc{f}$, is defined as the coefficient of the largest monomial in the support of $f$.

\section{Main Results}

We begin with the following theorem, which states a necessary and sufficient condition for the ascent of IDF property in polynomial extensions.

\begin{theorem}
\label{mainTheorem}
The following are equivalent for a domain $D$:
\begin{enumerate}
\item
$D$ is IDF and MCD-finite.
\item
$D[X]$ is IDF.
\item
For any set $\mybar{X}$ of indeterminates, $D[\mybar{X}]$ is IDF.
\end{enumerate}
\end{theorem}
\begin{proof}
$(1 \Longrightarrow 3)$
If $f \in D[\mybar{X}]$, then there exist indeterminates $X_1, \dotsc, X_n \in \mybar{X}$ such that $f \in D[X_1, \dotsc, X_n]$. Now, if $g \in \atoms{D[\mybar{X}]}$ and $g \aadd{D[\mybar{X}]} f$, then it is easy to see that $g \in \atoms{D[X_1, \dotsc, X_n]}$ and $g \aadd{D[X_1, \dotsc, X_n]} f$. Moreover, for any $g_1, g_2 \in D[X_1, \dotsc, X_n]$, $g_1 \assoc{D[\mybar{X}]} g_2$ if and only if $g_1 \assoc{D[X_1, \dotsc, X_n]} g_2$. Therefore, we only need to prove this implication for the case where $\mybar{X}$ is finite. So let $R = D[X_1, \dotsc, X_n]$ and suppose that $D$ is an IDF domain such that $R$ is not IDF. We show that $D$ is not MCD-finite.
 
Let $f \in \nunits{R}$ be such that there exists an infinite set $\{f_i\}_{i \in I}$ of non-associate irreducible divisors of $f$. Let $K$ be the field of fractions of $D$ and $R' = K[X_1, \dotsc, X_n]$. Since $R'$ is a UFD, there exists an infinite subset $J$ of $I$ such that the elements of $\{f_i\}_{i \in J}$ are associate in $R'$. Hence, there exists a set $S$ of monomials such that for every $i \in J$ the support of $f_i$ is $S$. There are two types of irreducible elements in $R$ with a support of size 1. First type, up to associates, consists of the elements $X_1, \dotsc, X_n$. Obviously, $f$ cannot have an infinite number of non-associate irreducible divisors of this type. The second type consists of every $a \in \atoms{D}$. If an element $a \in \atoms{D}$ divides $f$, then $a$ divides every single coefficient of $f$ in $D$, and since $D$ is IDF, $f$ cannot have an infinite number of non-associate irreducible divisors of the second type. Therefore, $2 \leq |S|$, and so since for every $i \in J$, $f_i \in \atoms{R}$, the GCD of the coefficients of each $f_i$ is equal to $1$.

Now, we fix a monomial order. For all $i, j \in J$, there exists an element $u_{i, j} \in \nonz{K}$ such that $f_i u_{i, j} = f_j$, and so $u_{i,j} = \frac{\lc{f_j}}{\lc{f_i}}$. Now, fix an element $t$ of $J$. Then for every $i \in J$,  $f_t = \left( \frac{\lc{f_t}}{\lc{f_i}} \right) f_i$, and so 
\[ \lc{f} f_t = \lc{f} \frac{\lc{f_t}}{\lc{f_i}}f_i . \]
Using the definition of monomial order, it is not difficult to see that $\lc{f_i} \aadd{D} \lc{f}$. Hence, $\lc{f} \left( \frac{\lc{f_t}}{\lc{f_i}} \right) \in \nonz{D}$. Also, if $i \neq j$, then 
\[ \lc{f}\frac{\lc{f_t}}{\lc{f_i}} \not{\assoc{D}} \lc{f}\frac{\lc{f_t}}{\lc{f_j}}  ; \] 
otherwise $\lc{f_i} \assoc{D} \lc{f_j}$, so $u_{i,j} \in \units{D}$, and so $f_i \assoc{R} f_j$, which is a contradiction. Since for every $i \in J$, the GCD of the coefficients of $f_i$ is equal to $1$, the set  
\[ \left\{\lc{f}\frac{\lc{f_t}}{\lc{f_i}}\right\}_{i \in J} \]
is an infinite set of non-associate maximal common divisors of the coefficients of $\lc{f}f_t$, and hence $D$ is not MCD-finite.

$(3 \Longrightarrow 2)$ This is obvious.

$(2 \Longrightarrow 1)$
If $D$ is not IDF, then obviously neither is $D[X]$.

Now, let both of the domains $D$ and $D[X]$ be IDF domains, but there exist $n \in \mathbb{N}$ and elements
$a_0, a_1, \dots, a_n$ in $\nonz{D}$ with infinitely many non-associate maximal common divisors $\{c_i\}_{i \in I}$.

Let $f = a_0 + a_1X + \dotsb + a_nX^n$. For every $i \in I$, there exists an $f_i \in \nunits{D[X]}$
such that $f = c_if_i$. Since the GCD of the coefficient of each $f_i$ is 1, no element of $D$ can appear in a factorization of $f_i$ into non-units. Hence, there exists a factorization of $f_i$ with maximum number (bounded by $\deg(f_i)$) of non-unit factors; say $f_i = h_1 \dotsm h_k$. It follows that each $h_i$ is necessarily an atom. Therefore, each $f_i$ has an atomic factorization in $D[X]$. But every irreducible divisor of each of the elements $f_i$ is an irreducible divisor of $f$ too. Therefore, since $D[X]$ is an IDF domain, there exists a finite set $\{g_1, \dots, g_m \}$ of non-associate irreducible divisors of $f$ such that every $f_i$, up to associates, is equal to a product of the elements $g_i$. Hence, for every $i \in I$, there exist $t_{1,i}, \dotsc, t_{m,i} \in \mathbb{N} \cup \{0\}$ such that 
\[f_i \assoc{D[X]} g_1^{t_{1,i}} \dotsm g_m^{t_{m,i}} .\]
Note that for $1 \leq j \leq m$, we have $1 \leq \deg(g_j)$. Hence, for $1 \leq j \leq m$, the set $\{t_{j,i}\}_{i \in I}$ is finite, and so there exist $k, \ell \in \mathbb{N}$ such that $k \neq \ell$ and 
\[ g_1^{t_{1,k}} \dotsm g_m^{t_{m,k}} = g_1^{t_{1,\ell}} \dotsm g_m^{t_{m,\ell}} , \]
and so $f_k \assoc{D[X]} f_\ell$. Hence, $c_k \assoc{D[X]} c_\ell$, and so $c_k \assoc{D} c_\ell$, which is a contradiction.
\end{proof}

This theorem entails the previously known results \cite[Theorem~1.9]{malcolmson2009polynomial} and \cite[Proposition~5.3]{anderson1990factorization} regarding sufficient conditions for the ascent of IDF property. Moreover, in the introduction of \cite{malcolmson2009polynomial}, it is mentioned and attributed to Muhammad Zafrullah that the polynomial extension of a domain that is both IDF and pre-Schreier is IDF. Hence, by Theorem \ref{mainTheorem}, being pre-Schreier must imply being MCD-finite. Indeed, we recall that a domain $D$ is pre-Schreier if and only if the poset of its principal ideals satisfies Riesz interpolation property, i.e., for every pair of finite subsets $A$ and $B$ of principal ideals of $D$, if $A \leq B$ (i.e., for every $\langle a \rangle \in A$ and $\langle b \rangle \in B$, $\langle a \rangle \subseteq \langle b \rangle$), then there exists a principal ideal $\langle x \rangle$ such that $A \leq \langle x \rangle \leq B$ (see \cite[Theorem~1.1]{Zafrullah1087}). Hence, every finite set of non-zero elements of a pre-Schreier domain, up to associates, has at most one MCD. So the class of pre-Schreier domains is a subclass of MCD-finite domains as expected.

Before going on, it is worthwhile to compare the notions of IDF and MCD-finite domains ideal-theoretically. A domain $D$ is IDF if and only if for any non-zero principal ideal $I$, there only exists a finite number of  ideals containing $I$ that are maximal with respect to being principal (see \cite[p.~12]{anderson1990factorization}). On the other hand, a domain $D$ is MCD-finite if and only if for any finitely generated ideal $I$, there only exists a finite number of  ideals containing $I$ that are minimal with respect to being principal.

Now, we are going to find a way to embed an arbitrary domain $D$ in a domain $R$ such that $R$ is not MCD-finite. The task of finding monoids that are not MCD-finite is more straightforward than finding domains that are not MCD-finite. Next, we give two examples of such monoids.

\begin{example}
\label{monoiddomainex}
\begin{enumerate}
\item
The additive monoid $C = \{x \in \mathbb{Q}\colon\;  1 \leq x\} \cup \{0\}$ is not MCD-finite. In fact, let $A$ be pairwise incomparable subset of $C$ (i.e., there do not exist $x, y \in A$ such that $1 \leq |x- y|$), where $2 \leq |A|$. Then if there exists an element $y$ in $A$ such that $y < 2$, then $\mathrm{MCD}(A) = \{0\}$. Another possibility is when all the elements of $A$ are strictly greater than $2$. In this case, $A$ has an infinite number of MCD's. Finally, if $2$ is the minimum element of $A$, then $\mathrm{MCD}(A) = \{1\}$.
\item
Let $S$ be a cancellative, torsion-free and reduced monoid and suppose that there exist $s,t \in S$ such that $\MCD(s, t)$ contains at least two (non-associate) elements $b$ and $d$ (An elementary example of such monoid is the additive monoid $(\mathbb{N} \cup \{0\}) \setminus \{1\}$. Note that $2$ and $3$ are both MCD's of $5$ and $6$).

Set $T \coloneqq \prod_{i \in \mathbb{N}}S$. Let $s'$ and $t'$ be the elements of $T$ with all the components equal to $s$ and $t$, respectively. Let $c_i$ be the element of $T$ with $b$ in its $i$th place, and $d$ in all the other places.
Then $\{c_i\}_{i \in \mathbb{N}}$ is a set of non-associate MCD of $s'$ and $t'$ and hence, $T$ is not MCD-finite. Moreover, it is easy to see that $T$ is reduced, cancellative and torsion-free.
\end{enumerate}
\end{example}

Now, using monoid domains, we can easily embed any domain into a domain which is not MCD-finite (another way for doing this is given in Example \ref{ReesEx}). Explicitly:

\begin{lemma}
\label{monoiddomain}
Let $D$ be a domain and let $S$ be a torsion-free, cancellative monoid that is not MCD-finite. Let $R = D[X; S]$.
Then $R$ is not MCD-finite, the extension $D \subseteq R$ is division-preserving and $\atoms{D} \subseteq \atoms{R}$. Also, in the special case where $S$ is reduced, we additionally have $\units{D} = \units{R}$.
\end{lemma}
\begin{proof}
The essential observation is that if $S$ is torsion-free and cancellative, then any divisor of a monomial of $R = D[X; S]$ is itself a monomial (see \cite[Theorem~11.1]{gilmer1984commutative}). All the parts can be deduced from this (see also \cite[Lemma~3.1]{Coykendall2011embedding}).
\end{proof}

Finally, we are going to embed a domain which is not MCD-finite into a domain that is IDF and has the same unit elements in such a way that the MCD's of elements of the original domain is preserved in the new domain (thus ensuring that the new domain is not MCD-finite). For doing this, we use a modified version of a powerful construction originally used by Roitman in \cite{roitman1993polynomial} (this technique was used to construct an atomic domain such that its polynomial extension is not atomic). This technique has also been used in \cite{Coykendall2011embedding}, \cite{Coykendall2004ap} and \cite{rand2015multiplicative}.

Let $D$ be a domain and $S \subseteq \nonz{D}$. Set
\[ \mathcal{L}(D;S) \coloneqq D[ \{X_s, \frac{s}{X_s}\colon\; s \in S \} ] .\]

Rand in \cite{rand2015multiplicative} proved the following result which we restate by adding the construction used in the proof.

\begin{theorem}[{{\cite[Theorem~2.7]{rand2015multiplicative}}}]
\label{RandTheorem}
Let $D$ be a domain and let $S$ be a subset of $\atoms{D}$ that is unit-closed. For every $n \in \mathbb{N} \cup \{0\}$, we define a domain $\mathcal{T}_n(D)$ inductively as follows:
\begin{enumerate}
\item
Set $\mathcal{T}_0(D;S) \coloneqq D$.
\item
For every $n \in \mathbb{N}$, set
$\mathcal{T}_n(D;S) \coloneqq \mathcal{L}(\mathcal{T}_{n-1}(D); \atoms{\mathcal{T}_{n-1}(D)} \setminus S)$.
\end{enumerate}
Set $\mathcal{T}^{\infty}(D;S)	 \coloneqq \bigcup_{i \in \mathbb{N}} \mathcal{T}_i (D;S)$. Then $\units{\mathcal{T}^{\infty}(D;S)} = \units{D}$ and $\atoms{\mathcal{T}^{\infty}(D;S)} = S$.
\end{theorem}

Now, we are ready to prove the stronger version of Theorem \ref{MOTheorem}.

\begin{theorem}
\label{MOstrong}
Let $D$ be a domain and let $S$ be a unit-closed subset of $\atoms{D}$. Then there exists a domain $R$ containing $D$ such that $R$ is not MCD-finite, $\atoms{R} = S$ and $D \subseteq R$ is division-preserving. 

In particular, if S does not contain an infinite set of non-associate elements, then R is an IDF domain.
\end{theorem}
\begin{proof}
Let $T$ be a cancellative, torsion-free and reduced monoid which is not MCD-finite (e.g., Example \ref{monoiddomainex}) and let $B = D[X, T]$. Then by Lemma \ref{monoiddomain}, $S \subseteq \atoms{B}$, $S$ is unit-closed  in $B$, $D \subseteq B$ is division-preserving and $B$ is not MCD-finite.

Now, if we set $R \coloneqq \mathcal{T}^{\infty}(B;S)$, then by Theorem \ref{RandTheorem}, $\atoms{R} = S$. Also, by using induction and by \cite[Lemma~3.2(1)]{roitman1993polynomial}, for every $n \in \mathbb{N}$, the extension $B \subseteq \mathcal{T}_n(B;S)$ is division-preserving, and so the extensions $B \subseteq R$ and hence $D \subseteq R$ are also division-preserving.

Now, let $V$ be a finite subset of $\nonz{B}$ such that $\mathrm{MCD}_B(V)$, up to associates, is infinite. Then by \cite[Lemma~3.2(6)]{roitman1993polynomial}, for every $n \in \mathbb{N}$, we have $\MCD_B(V) = \MCD_{{\mathcal{T}}_n}(V)$. The family $\{\mathcal{T}_i(B;S)\}_{i \in \mathbb{N}}$ satisfies the conditions of \cite[Lemma~3.1]{roitman1993polynomial}, and so by part 4 of the same lemma, $\MCD_B(V) = \MCD_{R}(V)$. Finally, since $\units{B} = \units{R}$, we conclude that $R$ is not MCD-finite.
\end{proof}

\begin{remark}
In Theorem \ref{MOstrong}, if $S$ is empty, then $R$ is an antimatter domain which is not MCD-finite, and hence by Theorem \ref{mainTheorem}, $R[x]$ is not an IDF domain. Therefore, this theorem entails Theorem \ref{MOTheorem} (although, a completely different construction is used here). In particular, Theorem \ref{MOstrong} shows that, as conjectured in \cite[Problem~1]{malcolmson2009polynomial}, the countability of the domain $D$ is superfluous in Theorem \ref{MOTheorem}.
\end{remark}

\section{More on MCD-finite domains}

We begin this section by a result on the behavior of MCD-finite domains under the $D + M$ construction. First, we mention some general properties of the $D + M$ constructions.

\begin{remark}
Let $T$ be a domain that can be written in the form $K+M$, where $K$ is a field and $M$ is a non-zero maximal ideal. Let $D$ be a subfield of $K$ and $R = D + M$. 

Any element of the form $m$ or $1+m$ (where $m \in M$) is an atom in $R$ if and only if it is an atom in $T$ (The case for $1 + m$ is proved in \cite[Lemma~1.5(i)]{Costa1986DM} and the case for $m$ can be proved similarly). Also, it is not difficult to see that for all $c_1, c_2 \in K$ and $0 \neq m \in M$, $c_1m \assoc{R} c_2m$ if and only if $c_1\nonz{D} = c_2\nonz{D}$. 

Let $x = k+m \in T$ where $m \in M$ and $k \in K$. For convenience, we use the following notation:
\[
\widehat{x} \coloneqq
\begin{cases}
x(=m) & k= 0 \\
1 + k^{-1}m & k \neq 0
\end{cases}
\]

It can easily be proved that for all $x, y \in T$, where $x \not\in M$, if $x \aadd{T} y$, then $\widehat{x} \aadd{R} \widehat{y}$.
\end{remark}

\begin{theorem}
\label{kmnonmcdfinite}
Let $T$ be a domain that can be written in the form $K+M$, where $K$ is a field and $M$ is a non-zero maximal ideal. Let $D$ be a subfield of $K$ and $R = D + M$. 
\begin{enumerate}
\item
Suppose that $M$ contains an atom. If $R$ is MCD-finite, then the group $\nonz{K}/\nonz{D}$ is finite and $T$ is MCD-finite.
\item
Suppose that every element of $M$, up to associates in $T$, has finitely many irreducible divisors in $M$. If $T$ is MCD-finite and the group $\nonz{K}/\nonz{D}$ is finite, then $R$ is MCD-finite.
\end{enumerate}
\end{theorem}
\begin{proof}
\begin{enumerate}
\item
Let $m \in M$ be an atom. Suppose that the group $\nonz{K}/\nonz{D}$ is infinite and let $\{c_i\}_{i \in I}$ be an infinite subset of $\nonz{K}$ such that $c_i\nonz{D} \neq c_j\nonz{D}$ for $i \neq j$.
Fix two elements $t$ and $v$ of $\nonz{K}$ such that $t\nonz{D} \neq v\nonz{D}$.
Then the set $\{c_im\}_{i \in I}$ is an infinite subset of non-associate divisors of $tm^2$ and $vm^2$ in $R$. On the other hand,  all the elements $\{\frac{tm^2}{c_im}\}_{i \in I}$ and $\{\frac{vm^2}{c_im}\}_{i \in I}$ are atoms in $R$ and moreover, for every $i \in I$,  
$ \frac{tm^2}{c_im} \not{\assoc{R}} \frac{vm^2}{c_im} ,$
 and so $\{c_im\}_{i \in I}$ is an infinite set of non-associate MCD's of $tm^2$ and $vm^2$. Therefore, $R$ is not an MCD-finite domain.

Now suppose on the contrary that $T$ is not MCD-finite. So assume that there exist non-associate elements $x_1, \dots, x_n \in \nunits{T}$ with infinite number of non-associate MCD's; say $\{ y_i \}_{i \in I}$. If for some $i \in I$, $y_i \not\in M$, then $\widehat{y_i} \aadd{R} \widehat{x_j}$ for $1 \leq j \leq n$ and $\widehat{y_i} \in \MCD_R(\widehat{x_1}, \dotsc, \widehat{x_n})$. So we may assume that $y_i \in M$ for every $i \in I$. Suppose that
\[x_j = (k_{i,j} + m_{i,j}) y_i \qquad k_{i,j} \in K, m_{i,j} \in M  .\]

Since $\nonz{K}/\nonz{D}$ is finite, we may assume that for every $1 \leq j \leq n$, the elements $k_{i,j}$ belong to the same coset. 
Hence, there exists $z_j \in \nonz{K}$ such that $y_i \aadd{R}  z_j x_j$ for every $i \in I$. Now
\[ \GCD_R \left( \frac{z_1 x_1}{y_i} , \dotsc, \frac{z_n x_n}{y_i} \right) = 1 ,\]
and so $\{ y_i \}_{i \in I}$ is also an infinite set of non-associate MCD's of $z_1 x_1 , \dotsc , z_n x_n$ in $R$, which is a contradiction.
\item
Suppose that $x_1, \dots, x_n \in \nunits{R}$ and, up to associates in $T$, $y_1, \dots, y_m$ are the MCD's of $x_1, \dotsc , x_n$ in $T$. 
 Suppose that $z \in \MCD_R(x_1, \dotsc, x_n)$.
If 
\[ \GCD_T \left( \frac{x_1}{z}, \dotsc , \frac{x_n}{z} \right) = 1, \] 
then $z \assoc{T} y_i$ for some $1 \leq i \leq m$. Now, assume that there exists $d \in \nunits{T}$ such that $d$ is a common divisor of $\frac{x_1}{z}, \dotsc , \frac{x_n}{z}$.
Note that $d$ must be in $M$ since otherwise $z\widehat{d}$ would be a common divisor of $x_1, \dotsc , x_n$ in $R$, which is a contradiction. For a similar reason, $d$ must be an atom and also an MCD of the elements $\frac{x_1}{z}, \dotsc , \frac{x_n}{z}$ in $T$. Hence, $ zd \assoc{T} y_i $ for some $1 \leq i \leq m$. Now, by the hypothesis, the set
\[ A= \left\{ \frac{y_i}{d} \in T\colon\; d \in M, d \in \atoms{T}, 1 \leq i \leq m \right\} \cup \{y_1, \dotsc, y_m \} ,\]
up to associates in $T$, is finite. We proved that for some $ [a] \in A / \!\!\!\assoc{T} $, $z \assoc{T} a$. Finally, if $c_1, \dotsc , c_k$ is a set of coset representatives of $\nonz{D}$ in $\nonz{K}$, then $z \assoc{R} c_i a$ for some $1 \leq i \leq k$. 
\end{enumerate}
\end{proof}
As we have already mentioned in the paragraph after Theorem \ref{mainTheorem}, pre-Schreier domains are MCD-finite, and in fact, any finite set of a pre-Schreier domain, up to associates, has at most one MCD. Hence, every pre-Schreier domain that is not a GCD domain (see, e.g., the paragraph after \cite[Theorem~2.4]{Cohn1968}) is an example of an MCD-finite domain that is not an MCD domain. Conversely, the domain $\mathbb{Q} + X\mathbb{R}[X]$ is a non-Noetherian ACCP (and hence MCD) domain that is not MCD-finite (we recall that for any field extension $F \subseteq K$, the group $\nonz{K}/\nonz{F}$ is finite if and only if $K = F$ or $K$ is finite (see \cite[Theorem~7]{brandis1965multiplikative})). In fact, even Noetherian domains are not necessarily MCD-finite; consider $\mathbb{R} + X\mathbb{C}[X]$. Also, $\mathbb{Z} + X\mathbb{Q}[[X]]$ is an example of a GCD (and hence MCD-finite) domain that is not IDF (The behavior of IDF, ACCP, GCD and Noetherian domains under the $D + M$ construction is studied in \cite[Proposition~4.3]{anderson1990factorization}, \cite[Proposition~1.2]{anderson1990factorization}, \cite[Theorem~11]{Brewer} and \cite[Theorem~4]{Brewer}).
 
 A trivial example of MCD-finite domains are FFD's. In particular, every Krull domain is an MCD-finite domain (see \cite[p.~14]{anderson1990factorization}). In fact, the stronger result \cite[Proposition~1]{Grams1975} also holds for MCD-finite domains. Actually, the proof of that theorem shows that any domain satisfying the hypothesis of the theorem has the following property:
\begin{center}
$(*)$ \text{The intersection of every infinite set of incomparable principal ideals is} 0 \end{center} and this property implies being both IDF and MCD-finite.

Now, we give an example of a domain with property $(*)$ that is neither pre-Schreier nor MCD domain. 

\begin{example}

Let $D$ be an FFD which is not a UFD. For example, let $F \subsetneq K$ be an extension of finite fields and let $D = F + XK[X]$. Then by \cite[Proposition~5.2]{anderson1990factorization}, $D$ is an FFD. The domain $D$ satisfies ACCP and hence is an MCD domain. Since UFD's are exactly GCD domains which satisfy ACCP, we conclude that $D$ is not a pre-Schreier domain.

Now, let $R = D[Z, \{ \frac{X}{Z^n}, \frac{Y}{Z^n}\colon\; n \in \mathbb{N} \cup \{0\}] $. This domain is not an MCD domain and in fact the elements $X$ and $Y$ do not have any MCD's (see \cite[Example~5.1]{roitman1993polynomial}). But $R$ satisfies property $(*)$. To see this, suppose on the contrary that $\{ c_iR \}_{i \in I}$ is an infinite set of incomparable principal ideals of $R$ such that $\bigcap_{i \in I} c_iR \neq 0$. By \cite[Example~2]{Anderson1996FFD}, $D[Z, \frac{1}{Z}]$ is an FFD and hence there exist $i, j \in I$ such that $i \neq j$ and $c_i \assoc{T} c_j$, where $T = (D[Z, \frac{1}{Z}])[X, Y]$. But $c_i \not{\assoc{R}} c_j$ and hence there exists an $n \in \mathbb{N} \cup \{0\}$ such that $c_i = Z^nc_j$, which contradicts the incomparability assumption. Therefore, $R$ satisfies property $(*)$, hence is both IDF and MCD-finite. But it is neither pre-Schreier nor MCD domain.
\end{example}

Finally, we mention another way to embed a domain $D$ into a domain $R$ which is not MCD-finite. This time, for a finite non-singleton subset $A$ of $\nunits{D}$ that does not generate $D$, we construct the domain $R$ in such a way that the set $\mathrm{MCD}_R(A)$, up to associates, becomes infinite.

\begin{example}
\label{ReesEx}
Let $D$ be a domain and let $c_1, \dotsc, c_n$ be non-associate elements in $\nunits{D}$ such that $2 \leq n$ and $D \neq \langle c_1, \dotsc , c_n \rangle$. For every $i \in \mathbb{N} \cup \{0\}$, define the domain $R_i$ by induction as follows:
\begin{enumerate}
\item
Set $R_0 \coloneqq D$ and $I_0 \coloneqq \langle c_1, \dotsc, c_n \rangle_{R_0}$.
\item
For $i \in \mathbb{N}$, set $R_i \coloneqq R_{i-1}[X_i, \{ \frac{a}{X_i}\colon\; a \in I_{i-1} \}]$ and $I_i \coloneqq \langle c_1, \dotsc, c_n \rangle_{R_i}$.
\end{enumerate}
Note that for any $n \in \mathbb{N}$, $I_n \neq R_n$, and hence by \cite[Lemma~2.4]{roitman1993polynomial} and \cite[Lemma~2.10]{roitman1993polynomial}, for every $i \in \mathbb{N}$,
\[ X_1, \dotsc, X_i \in \mathrm{MCD}_{R_i}(c_1, \dotsc, c_n) .\]
Set $R \coloneqq  \bigcup_{i \in \mathbb{N}} R_i$. By \cite[Lemma~2.3(1)]{roitman1993polynomial}, for all $i, j \in \mathbb{N} \cup \{0 \}$ with $i \leq j$, the extension $R_i \subseteq R_j$ is division-preserving, and hence by \cite[Lemma~3.1(3)]{roitman1993polynomial}, \[\{X_i\}_{i \in \mathbb{N}} \subseteq \mathrm{MCD}_R(c_1, \dotsc, c_n) ,\]
and so $R$ is not MCD-finite.
\end{example}

\noindent \textbf{Acknowledgement.} The authors would like to thank the referee whose careful reading and valuable comments improved the paper.

\bibliography{Manuscript}

\begin{bibdiv}
\begin{biblist}

\bib{anderson1990factorization}{article}{
      author={Anderson, D.~D.},
      author={Anderson, David~F.},
      author={Zafrullah, Muhammad},
       title={Factorization in integral domains},
        date={1990},
        ISSN={0022-4049},
     journal={J. Pure Appl. Algebra},
      volume={69},
      number={1},
       pages={1\ndash 19},
}

\bib{Anderson1996FFD}{article}{
      author={Anderson, D.~D.},
      author={Mullins, Bernadette},
       title={Finite factorization domains},
        date={1996},
        ISSN={0002-9939},
     journal={Proc. Amer. Math. Soc.},
      volume={124},
      number={2},
       pages={389\ndash 396},
}

\bib{brandis1965multiplikative}{article}{
      author={Brandis, Albrecht},
       title={\"{U}ber die multiplikative {S}truktur von
  {K}\"orpererweiterungen},
        date={1965},
        ISSN={0025-5874},
     journal={Math. Z.},
      volume={87},
       pages={71\ndash 73},
}

\bib{Brewer}{article}{
      author={Brewer, J.~W.},
      author={Rutter, E.~A.},
       title={{$D+M$} constructions with general overrings},
        date={1976},
        ISSN={0026-2285},
     journal={Michigan Math. J.},
      volume={23},
      number={1},
       pages={33\ndash 42},
}

\bib{Cohn1968}{article}{
      author={Cohn, P.~M.},
       title={Bezout rings and their subrings},
        date={1968},
     journal={Proc. Cambridge Philos. Soc.},
      volume={64},
       pages={251\ndash 264},
}

\bib{Costa1986DM}{article}{
      author={Costa, D.},
      author={Mott, J.},
      author={Zafrullah, M.},
       title={Overrings and dimensions of general {$D+M$} constructions},
        date={1986},
        ISSN={0022-2941},
     journal={J. Natur. Sci. Math.},
      volume={26},
      number={2},
       pages={7\ndash 13},
}

\bib{Coykendall1999antimatter}{article}{
      author={Coykendall, Jim},
      author={Dobbs, David~E.},
      author={Mullins, Bernadette},
       title={On integral domains with no atoms},
        date={1999},
        ISSN={0092-7872},
     journal={Comm. Algebra},
      volume={27},
      number={12},
       pages={5813\ndash 5831},
}

\bib{Coykendall2011embedding}{article}{
      author={Coykendall, Jim},
      author={Mammenga, Brenda~J.},
       title={An embedding theorem},
        date={2011},
        ISSN={0021-8693},
     journal={J. Algebra},
      volume={325},
       pages={177\ndash 185},
}

\bib{Coykendall2004ap}{article}{
      author={Coykendall, Jim},
      author={Zafrullah, Muhammad},
       title={A{P}-domains and unique factorization},
        date={2004},
        ISSN={0022-4049},
     journal={J. Pure Appl. Algebra},
      volume={189},
      number={1-3},
       pages={27\ndash 35},
}

\bib{gilmer1984commutative}{book}{
      author={Gilmer, Robert},
       title={Commutative semigroup rings},
      series={Chicago Lectures in Mathematics},
   publisher={University of Chicago Press, Chicago, IL},
        date={1984},
        ISBN={0-226-29391-2; 0-226-29392-0},
}

\bib{Grams1975}{article}{
      author={Grams, Anne},
      author={Warner, Hoyt},
       title={Irreducible divisors in domains of finite character},
        date={1975},
        ISSN={0012-7094},
     journal={Duke Math. J.},
      volume={42},
       pages={271\ndash 284},
}

\bib{herzog2011monomial}{book}{
      author={Herzog, J\"urgen},
      author={Hibi, Takayuki},
       title={Monomial ideals},
      series={Graduate Texts in Mathematics},
   publisher={Springer-Verlag London, Ltd., London},
        date={2011},
      volume={260},
        ISBN={978-0-85729-105-9},
}

\bib{malcolmson2009polynomial}{article}{
      author={Malcolmson, P.},
      author={Okoh, F.},
       title={Polynomial extensions of {IDF}-domains and of {IDPF}-domains},
        date={2009},
        ISSN={0002-9939},
     journal={Proc. Amer. Math. Soc.},
      volume={137},
      number={2},
       pages={431\ndash 437},
}

\bib{rand2015multiplicative}{article}{
      author={Rand, Ashley},
       title={Multiplicative subsets of atoms},
        date={2015},
        ISSN={1306-6048},
     journal={Int. Electron. J. Algebra},
      volume={18},
       pages={107\ndash 116},
}

\bib{roitman1993polynomial}{article}{
      author={Roitman, Moshe},
       title={Polynomial extensions of atomic domains},
        date={1993},
        ISSN={0022-4049},
     journal={J. Pure Appl. Algebra},
      volume={87},
      number={2},
       pages={187\ndash 199},
}

\bib{Zafrullah1087}{article}{
      author={Zafrullah, Muhammad},
       title={On a property of pre-{S}chreier domains},
        date={1987},
        ISSN={0092-7872},
     journal={Comm. Algebra},
      volume={15},
      number={9},
       pages={1895\ndash 1920},
}

\end{biblist}
\end{bibdiv}

\end{document}